\documentclass{amsproc}

\newtheorem{theorem}{Theorem}[section]
\newtheorem{lemma}[theorem]{Lemma}
\newtheorem{proposition}{Proposition}[section]
\theoremstyle{definition}
\newtheorem{definition}[theorem]{Definition}
\newtheorem{corollary}[theorem]{Corollary}

\newtheorem{example}[theorem]{Example}

\theoremstyle{remark}
\newtheorem{remark}[theorem]{Remark}

\numberwithin{equation}{section}

\begin{document}

\title{Symmetric cohomology of groups in low dimension}

%    Information for first author
\author{Mihai D. Staic}
%    Address of record for the research reported here
\address{Department of Mathematics, Indiana University, Rawles Hall, Bloomington, IN 47405, USA }
\address{Institute of Mathematics of the Romanian Academy, PO.BOX 1-764, RO-70700 Bu\-cha\-rest, Romania}
%} 
\thanks{Research partially supported by the CNCSIS project {\it Hopf algebras, cyclic homology and monoidal categories} contract no. 560/2009.}
%    Current address
\curraddr{}
\email{mstaic@indiana.edu}

%    \thanks will become a 1st page footnote.
%\thanks{The first author was supported in part by NSF Grant \#000000.}

%    Information for second author
%\author{Author Two}
%\address{Mathematical Research Section, School of Mathematical Sciences,
%Australian National University, Canberra ACT 2601, Australia}
%\email{two@maths.univ.edu.au}

%    General info
%\subjclass{Primary ...}
\date{January 1, 1994 and, in revised form, June 22, 1994.}

%\dedicatory{I thank to my advisor Prof. Samuel D. Schack for  }

%\keywords{symmetric cohomology, group extensions}

\begin{abstract} We give an explicit characterization for group extensions that  co\-rres\-pond to elements of the symmetric  cohomology $HS^2(G,A)$. We also give conditions for the map $HS^n(G,A)\to H^n(G,A)$ to be injective.
\end{abstract}

\maketitle

%\section*{This is an unnumbered first-level section head}

%
%%%%%%%%%%%%%%%%%%%%%%%%%%%%%%%%%%%%%%%%%%%%%%%%%%%
%\begin{center}
%\\
%\end{center}

%\begin{minipage}{4.5in}\footnotesize\baselineskip=10pt
%	\centerline{ABSTRACT} 
%	\parindent=15pt #1\par 
%\parindent=15pt #2\par

%	\parindent=15pt #3\par
%	\parindent=15pt #4\par
%	\end{minipage}}

%%%%%%%%%%%%%%%%%%%%%%%

\section*{Introduction}
%%%%%%%%%%%%%%%%%%%%%%%

Motivated by a topological construction, in \cite{sm} we constructed an action of the symmetric group $\Sigma_{n+1}$ on $C^n(G,A)$ ($G$ is a group and $A$ is a $G$-module). We proved that this action gives a subcomplex $CS^n(G,A)=C^n(G,A)^{\Sigma_{n+1}}$ of the usual cohomology complex,  in particular we have a new cohomology theory $HS^n(G,A)$, called symmetric cohomology, and a map $HS^n(G,A)\to H^n(G,A)$. We also showed that to a topological space $M$ with no elements of order 2 or 3 in $\pi_1(M)$, one can associate an element $\alpha\in HS^3(\pi_1(M),\pi_2(M))$. The image of $\alpha$ in $H^3(\pi_1(M),\pi_2(M))$ is the classical $k$-invariant introduced by Eilenberg and MacLane in \cite{em} (also called Postnikov invariant).

In the case $n=2$ the map $HS^2(G,A)\to H^2(G,A)$  is always injective. A well known result in group theory states that the elements of $H^2(G,A)$ are in bijection with the extensions of $G$ by $A$. So, it is a natural question to ask what are the extensions that correspond to elements of the symmetric cohomology. In this paper we prove that the elements of $HS^2(G,A)$ are in  bijection with those extensions that admit a section $s$ with the property that $s(g^{-1})=s(g)^{-1}$. As a corollary we get that $HS^2(G,A)\cong H^2(G,A)$ if $G$ has no elements of order 2. 

There are similarities between the symmetric cohomology defined here  and the homology theory defined for crossed simplicial groups in \cite{fl}. However  the results in \cite{fl} do not apply to our setting. In particular it is not clear what is the relation between symmetric cohomology and the ordinary one.  We approach here this problem and give conditions on the group $A$  for the map $HS^n(G,A)\to H^n(G,A)$ to be injective.

\section{Preliminaries}

We recall from  \cite{kb} some results and notations about cohomology and group extensions.

Let $G$ be a group and $A$ a $G$-module. We set  $C^n(G,A)=\{\sigma : G^n\to A\}$ and define $\partial_n: C^n(G,A)\to C^{n+1}(G,A)$ by
\begin{eqnarray*}
\partial_n(\sigma)(g_1,...,g_{n+1})=g_1\sigma(g_2,...,g_{n+1})
-\sigma(g_1g_2,g_3,...,g_{n+1})+...+\\
+(-1)^n\sigma(g_1,...,g_ng_{n+1})+(-1)^{n+1}\sigma(g_1,...,g_n).
\end{eqnarray*} 
 Define $d^j:C^n(G,A)\to C^{n+1}(G,A)$ by 
\begin{eqnarray*}
&&d^0(\sigma)(g_1,...,g_{n+1})=g_1\sigma(g_2,...,g_{n+1}),\\
&&d^j(\sigma)(g_1,...,g_{n+1})=\sigma(g_1,...,g_jg_{j+1},...,g_{n+1}) \; {\rm for}\; 1\leq j \leq n,\\
&&d^{n+1}(\sigma)(g_1,...,g_{n+1})=\sigma(g_1,...,g_n).
\end{eqnarray*}
Let's notice that $\partial_n(\sigma)=\sum_0^{n+1}(-1)^jd^j$.
It is well known that in this way we obtain a chain complex and its homology groups are denoted by $H^n(G,A)$. 
\begin{example}
A map $\sigma:G\times G\to A$ is a 2-cocycle if: $g\sigma(h,k)-\sigma(gh,k)+\sigma(g,hk)-\sigma(g,h)=0$, and $\sigma$ is a 2-coboundary if there exist a map $\psi: G\to A$ such that:
$\sigma(g,h)=g\psi(h)-\psi(gh)+\psi(g)$.
\end{example}
In \cite{sm} it was constructed an action of $\Sigma_{n+1}$ on $C^n(G,A)$ (for every $n$) and  it was proved that it is compatible with the differential. 

It is enough to say what is the action of the transpositions $\tau_i=(i,i+1)$ for $1\leq i\leq n$. For $\sigma \in C^n(G,A)$ we define:
\begin{eqnarray*}
&&(\tau_1\sigma)(g_1,g_2,g_3,...,g_n)=-g_1\sigma((g_1)^{-1},g_1g_2,g_3,...,g_n),\\
&&(\tau_i\sigma)(g_1,g_2,g_3,...,g_n)=-\sigma(g_1,...,g_{i-2},g_{i-1}g_i,(g_i)^{-1},g_ig_{i+1},g_{i+2},...,g_n),\\
&&\; {\rm for }\; 1<i<n,\\
%&&((n-1,n)\sigma)(g_1,g_2,g_3,...,g_n)=-\sigma(g_1,g_2,..., g_{n-2}g_{n-1},(g_{n-1})^{-1},g_{n-1}g_n),\\
&&(\tau_n\sigma)(g_1,g_2,g_3,...,g_n)=-\sigma(g_1,g_2,g_3,...,g_{n-1}g_n,
(g_n)^{-1}).
\end{eqnarray*} 
\begin{proposition}(see \cite{fl} and \cite{sm}) The above formulas define an action of $\Sigma_{n+1}$ 
on $C^n(G,A)$ which is compatible with the differential $\partial$.
\label{scg}
\end{proposition}
\begin{definition} The subcomplex of the invariants is called the symmetric cochain and is denoted  $CS^n(G,A)=C^n(G,A)^{\Sigma_{n+1}}$. Its homology is called the symmetric cohomology of $G$ with coefficients in A  and is denoted 
$HS^n(G,A)$.
\end{definition}
\begin{remark}
There is a natural map from $HS^n(G,A)$ to $H^n(G,A)$.
\end{remark}
\begin{example} 
One can see that $\psi:G\to A$ is symmetric if $\psi(g)=-g\psi(g^{-1})$, and  
$\sigma:G\times G\to A$ is symmetric if: $\sigma(g,h)=-g\sigma(g^{-1},gh)=-\sigma(gh,h^{-1})$. 
\end{example}
The following computation was done by S. Van Ault \cite{sva} using GAP.
\begin{example}  We denote by $\mathbb{Z}_n$ the cyclic group of order n. With the trivial $\mathbb{Z}_2$ (respectively $\mathbb{Z}_4$) action on $\mathbb{Z}$ we have: $H^2(\mathbb{Z}_2,\mathbb{Z})=\mathbb{Z}_2$, $HS^2(\mathbb{Z}_2,\mathbb{Z})=0$, $H^2(\mathbb{Z}_4,\mathbb{Z})=\mathbb{Z}_4$ and  $HS^2(\mathbb{Z}_4,\mathbb{Z})=\mathbb{Z}_2$.
\end{example}

By an extension of $G$ by $A$ we mean a short exact sequence:
 \begin{eqnarray}0\rightarrow  A\stackrel{i}{\rightarrow} X  \stackrel{\pi}{\rightarrow} G\rightarrow  0.\label{extension}
 \end{eqnarray}
Two extensions are equivalent if there exists a morphism $f:X\to X'$ such that $fi=i'$ and $\pi 'f=\pi$. To a section $t:G\to X$ of the extension (\ref{extension}) one can associate a 2-cocycle:
\begin{eqnarray}
\sigma(g,h)=i^{-1}(t(g)t(h)t(gh)^{-1}).
\end{eqnarray}

\begin{theorem} (see \cite{kb}) There is a one to one correspondence between the elements of $H^2(G,A)$ and the set of equivalence classes of extensions of $G$  by $A$ \label{cohomologygroup}.
\end{theorem}

\section{Symmetric Extensions}

First we prove a result which was stated in \cite{sm}.
\begin{lemma} The map $HS^2(G,A)\to H^2(G,A)$ is one to one. \label{symmonetoone} 
\end{lemma}
\begin{proof} 
If $\sigma \in (ZS)^2(G,A) \cap B^2(G,A)$ then:
$$\sigma(g,h)=g\psi(h)-\psi(gh)+\psi(g)$$
Also $\sigma$ is symmetric:
\begin{equation*}
\sigma(g,h)=-g\sigma(g^{-1},gh)=-\psi(gh)+g\psi(h)-g\psi(g^{-1})
\end{equation*}
\begin{equation*}
\sigma(g,h)=-\sigma(gh,h^{-1})=-gh\psi(h^{-1})+\psi(g)-\psi(gh)
\end{equation*}
Which means that:
$$\psi(g)=-g\psi(g^{-1})$$ 
and so $\sigma\in BS^2(G,A)$. 
\end{proof} 

Let  $\sigma$ be a 2-cocyle.  On $X=A\times G$ we have the following multiplication:
\begin{equation}
(a,g)(b,h)=(a+gb+\sigma(g,h),gh)
\end{equation}
Define a section $s:G\to X$ by $s(g)=(0,g)$. 
Suppose that $\sigma$ is a symmetric cocycle, then we have:
\begin{equation*}
\sigma(g,h)+g\sigma(g^{-1},gh)=0\\
\end{equation*} 
\begin{equation*}
\sigma(g,h)+\sigma(gh,h^{-1})=0
\end{equation*} 
Let's notice that:
\begin{equation*}
s(g)(s(g^{-1})s(gh))=(g\sigma(g^{-1},gh)+\sigma(g,h),gh)=(0,gh)=s(gh)
\end{equation*} 
\begin{equation*}
(s(gh)s(h^{-1}))s(h)=(\sigma(g,h)+\sigma(gh,h^{-1}),gh)=(0,gh)=s(gh)
\end{equation*} 
And so we get:
$$s(g^{-1})=s(g)^{-1}$$
Suppose that we have a section $t:G\to X$ of (\ref{extension}) such that 
 \begin{eqnarray}t(g^{-1})=t(g)^{-1}.\label{symmsection}
 \end{eqnarray}
To this section we associate a 2-cocylce 
\begin{eqnarray}\sigma(g,h)=i^{-1}(t(g)t(h)t(gh)^{-1}).\label{symmcocycle}
\end{eqnarray}
We want to prove that $\sigma$ a symmetric cocycle. Indeed we have:
\begin{eqnarray*}
g\sigma(g^{-1},gh)&=&gi^{-1}(t(g^{-1})t(gh)t(h^{-1}))\\
&=&i^{-1}(t(g)((t(g^{-1})t(gh)t(h^{-1}))t(g^{-1}))\\
&=&i^{-1}(t(gh)t(h)^{-1}t(g)^{-1})\\
&=&i^{-1}((t(g)t(h)t(gh)^{-1})^{-1})\\
&=&-i^{-1}(t(g)t(h)t(gh)^{-1})\\
&=&-\sigma(g,h).
\end{eqnarray*}
\begin{eqnarray*}
\sigma(gh,h^{-1})&=&i^{-1}(t(gh)t(h^{-1})t(g^{-1}))\\
&=&i^{-1}((t(g)t(h)t(gh)^{-1})^{-1})\\
&=&-i^{-1}(t(g)t(h)t(gh)^{-1})\\
&=&-\sigma(g,h).
\end{eqnarray*}

Take now two sections $t$ and $u$ of $\pi$ which satisfy (\ref{symmsection}). To these sections we  associate 2-cocycles $\sigma_1$ and $\sigma_2$  as in (\ref{symmcocycle}). These 
2-cocycles are symmetric and cohomologicaly equivalent. From  Lemma \ref{symmonetoone} we see that $\sigma_1$ and $\sigma_2$ give the same class in $HS^2(G,A)$.  To conclude from the above discussion we have:
\begin{theorem} There is a one to one correspondence between the elements of $HS^2(G,A)$ and the set of extensions that admit a section $t$ which satisfy (\ref{symmsection}). \label{main}
\end{theorem} 
\begin{proof} It follows from above discussion.
\end{proof}
\begin{corollary} If $G$ is a group with no elements of order 2 then 
$HS^2(G,A)\cong H^2(G,A)$.
\end{corollary}
\begin{proof} If $G$ is a group with no element of order 2 then any extension has a section which satisfy (\ref{symmsection}). The corollary follows now from  Theorem \ref{cohomologygroup} and  Theorem \ref{main}.
\end{proof}
%%%%%%%%%
\begin{example}There are two extensions of $\mathbb{Z}_2$ by $\mathbb{Z}$:
$$0\rightarrow  \mathbb{Z}\rightarrow \mathbb{Z}\times \mathbb{Z}_2  \rightarrow \mathbb{Z}_2 \rightarrow  0$$
$$0\rightarrow  \mathbb{Z}\rightarrow \mathbb{Z} \rightarrow \mathbb{Z}_2 \rightarrow  0.$$
The first one admits a section which satisfies (\ref{symmsection}) the second one does not. This corresponds to the fact that $H^2(\mathbb{Z}_2,\mathbb{Z})=\mathbb{Z}_2$ and $HS^2(\mathbb{Z}_2,\mathbb{Z}_n)=0$.
\end{example}

\begin{example}There are four  nonequivalent extensions of $\mathbb{Z}_4$ by $\mathbb{Z}$. 
\begin{eqnarray}
0\rightarrow \mathbb{Z}\stackrel{i}{\rightarrow} \mathbb{Z}\times \mathbb{Z}_4  \stackrel{\pi}{\rightarrow} \mathbb{Z}_4\rightarrow  0,
\end{eqnarray}
with $i(x)=(x,0)$, $\pi(a,\widehat{b})=\widehat{b}$. Take a  section  $t(\widehat{x})=(0,\widehat{x})$. 
\begin{eqnarray} 
0\rightarrow  \mathbb{Z}\stackrel{j}{\rightarrow} \mathbb{Z}\times \mathbb{Z}_2  \stackrel{\delta}{\rightarrow} \mathbb{Z}_4\rightarrow  0,
\end{eqnarray}
with $j(x)=(2x,\overline{x})$, $\delta(a,\overline{b})=\widehat{a+2b}$. For this  extension we can take a section  $s$ defined by $s(\widehat{0})=(0,\overline{0})$, $s(\widehat{1})=(-1,\overline{1})$,  $s(\widehat{2})=(0,\overline{1})$ and $s(\widehat{3})=(1,\overline{1})$.
\begin{eqnarray} 
0\rightarrow  \mathbb{Z}\stackrel{k}{\rightarrow} \mathbb{Z}  \stackrel{\epsilon}{\rightarrow} \mathbb{Z}_4\rightarrow  0,
\end{eqnarray}
where $k(x)=-4x$ and $\epsilon(a)=\widehat{a}$.
\begin{eqnarray}
0\rightarrow  \mathbb{Z}\stackrel{l}{\rightarrow} \mathbb{Z}  \stackrel{\gamma}{\rightarrow} \mathbb{Z}_4\rightarrow  0,
\end{eqnarray}
where $l(x)=4x$ and $\gamma(a)=\widehat{a}$. 
As one can see the first two extensions admit a section which satisfy (\ref{symmsection}) the last two  do not. This corresponds to the fact that $H^2(\mathbb{Z}_4,\mathbb{Z})=\mathbb{Z}_4$ and $HS^2(\mathbb{Z}_4,\mathbb{Z})=\mathbb{Z}_2$.
\end{example}

\begin{example} Let $B_n$ be the braid group and $P_n$ the pure braid group. Consider the  extension 
$$0\to P_n/[P_n,P_n]\to B_n/[P_n,P_n]\to \Sigma_n\to 0$$
One can show that there is no element of order two in $B_n/[P_n,P_n]$ whose image in $\Sigma_n$ is the transposition $\tau_1=(1,2)$. In particular this means that the above extension is not symmetric.
\end{example}

%%%%%%%%%%%%%%%%%%%%%%%%%%%%%
%%%%%%%%%%%%%%%%%%%%%%%%%%%%%%%%%%%%%%%

%%%%%%%%

\section{Injectivity}

We denote by $\tau_i^{(n+1)}$ the transposition $(i,i+1)\in \Sigma_{n+1}$. We have the following relations between the action of $\Sigma_{n+1}$,  $\Sigma_{n+2}$ and the maps $d^j:C^n(G,A)\to C^{n+1}(G,A)$. 
\begin{eqnarray}
&&\tau_i^{(n+2)}d^j=d^j\tau_i^{(n+1)}\;  {\rm if} \; i< j
\end{eqnarray}
\begin{eqnarray}
&&\tau_i^{(n+2)}d^j=d^j\tau_{i-1}^{(n+1)} \; {\rm if} \; j+2\leq i \label{d_0}
\end{eqnarray}
\begin{eqnarray}
&&\tau_i^{(n+2)}d^{i-1}=-d^i\label{d_10}
\end{eqnarray}
\begin{eqnarray}
&&\tau_i^{(n+2)}d^i=-d^{i-1}
\end{eqnarray}
We denote by $\Sigma_{n+1}(k:n+k)$  the group of permutations for  the symbols $k,k+1,...,n+k$. As a convention we write $\Sigma_{n+1}$ for  $\Sigma_{n+1}(1:n+1)$ Define 
$$S_{n+1}(k:n+k)=(\sum_{\sigma\in \Sigma_{n+1}(k:n+k)}\sigma)\in {\bf Z}[\Sigma_{n+1}(k:n+k)]$$
By counting the elements in the left and right sum one can notice that 
$$S_{n+2}(1:n+2)=(1+\tau_1+\tau_2\tau_1+...+\tau_{n+1}\tau_n ...\tau_2\tau_1)S_{n+1}(2:n+2).$$
Also  using \ref{d_10} one gets:
\begin{eqnarray*}
&&\tau_1d^0=-d^1,\\
&&\tau_2\tau_1d^0=-\tau_2d^1=d^2\\
&&...\\
&&\tau_n\tau_{n-1}...\tau_2\tau_1d^0=(-1)^nd^n \label{d_0n}
\end{eqnarray*}
 which give: 
\begin{eqnarray*}
\partial_n&=&d^0-d^1+d^2-...+(-1)^{n+1}d_{n+1}\\
&=&d^0+\tau_1d_0+\tau_2\tau_1d^0+...+\tau_{n+1}\tau_n...\tau_2\tau_1d^0
\end{eqnarray*}
Using the above equality and (\ref{d_0}) we get:
\begin{eqnarray*}
\partial_nS_{n+1}&=&(1+\tau_1+\tau_2\tau_1+...+\tau_{n+1}...\tau_2\tau_1)d^0S_{n+1}(1:n+1)\\
&=&(1+\tau_1+\tau_2\tau_1+...+\tau_{n+1}...\tau_2\tau_1)S_{n+1}(2:n+2)d^0\\
&=&S_{n+2}(1:n+2)d^0\\
&=&S_{n+2}d^0
\end{eqnarray*}
Also one can notice that:
\begin{eqnarray*}
S_{n+2}d^0=S_{n+2}(-\tau_1)d^1=-S_{n+2}d^1
\end{eqnarray*}
or more generally:
\begin{eqnarray}
S_{n+2}d^0=-S_{n+2}d^1=S_{n+2}d^2=...=
(-1)^{n+1}S_{n+2}d^{n+1}
\end{eqnarray}
If we add everything together we get:
\begin{eqnarray}
(n+2)\partial_nS_{n+1}=S_{n+2}\partial_n \label{eqc}
\end{eqnarray}

\begin{proposition}
Suppose that $A$ is a group such that $n+1$ is not a zero divisor  and the equation  $n!x=a$ has exactly one solution.  Then the natural map $i:HS^n(G,A)\to H^n(G,A)$ is one to one. 
\end{proposition}
\begin{proof}
Take $\alpha\in ZS^n(G,A)$ such that $\alpha=\partial_{n-1}(\beta)$ for some $\beta \in C^{n-1}(G,A)$. Since $\alpha$ is symmetric $S^{n+1}(\alpha)=(n+1)!\alpha$ and so:
\begin{eqnarray*}
(n+1)!\alpha &=&S_{n+1}\alpha\\
&=&S_{n+1}\partial_{n-1}(\beta)\\
&=&(n+1)\partial_{n-1}(S_n(\beta))
\end{eqnarray*}
this gives
$$\alpha=\partial_{n-1}(\frac{1}{n!}S^n(\beta))$$
which means that the map $i$ is injective.
\end{proof}
%%%%%%%%
%\bibitem{dr} 
%Drinfeld, V. G. Quantum groups. Proc. ICM Berkeley 1986,  
%AMS Providence, R.I. {\bf 1987}, vol 1, 798-820.
\begin{remark} There is another way to define the cohomology of groups. One can take the differential $\widetilde{\partial_n}:\widetilde{C}^n(G,A)\to \widetilde{C}^{n+1}(G,A)$, 
\begin{eqnarray*} \widetilde{\partial_n}(\sigma)(g_0,g_1,g_2,...,g_n)&=&g_0\sigma(g_0^{-1}g_1,g_0^{-1}g_2,...,g_0^{-1}g_n)-\sigma(g_1,g_2,...,g_n)+\\
&&+\sigma(g_0,g_2,...,g_n)-...+(-1)^{n+1}\sigma(g_0,g_1,...,g_{n-1})
\end{eqnarray*}
If we define $j_n:C^n(G,A)\to \widetilde{C}^n(G,A)$,   
$$j_n(\sigma)(g_1,g_2,...,g_n)=\sigma(g_1,g_1g_2,...,g_1g_2...g_n)$$ 
one can show that $j_n\partial_n=\widetilde{\partial_n}j_n$ and so $\widetilde{\partial_n}$ defines the same cohomology. It was pointed to us by Z. Fiedorowicz  that the corresponding action of $\Sigma_{n+1}$ is:
\begin{eqnarray*}
&&(\tau_1\sigma)(g_1,g_2,...,g_n)=-g_1\sigma(g_1^{-1},g_1^{-1}g_2,...,g_1^{-1}g_n)\\ 
&&(\tau_2\sigma)(g_1,g_2,...,g_n)=-\sigma(g_2,g_1,g_3,...,g_n)\\
&&...\\
&&(\tau_n\sigma)(g_1,...,g_{n-2},g_{n-1},g_n)=-\sigma(g_1,...,g_{n-2},g_n,g_{n-1})
\end{eqnarray*}
\end{remark}

\section*{Acknowledgment}
I would like to thank Professor Samuel D. Schack for many useful discussions. 
%%%%%%%%%%
%%%%%%%%%%%%%%%5
%%%%%%%%%%%%%%%%%%%%%%%%%%%%%%%%%%%%5

\bibliographystyle{amsalpha}

\end{document}